\documentclass[a4paper]{amsart}

\usepackage{amsmath,amsthm,amsfonts,amssymb,amscd}
\usepackage{enumerate}

\input{xy}
\xyoption{all}

\newtheorem{teo}{Theorem}[section]
\newtheorem{pro}[teo]{Proposition}
\newtheorem{lem}[teo]{Lemma}
\newtheorem{cor}[teo]{Corollary}
\theoremstyle{definition}

\theoremstyle{remark}
\newtheorem{rem}[teo]{Remark}
\newtheorem{exa}[teo]{Example}

\numberwithin{equation}{section}

\newenvironment{sis}{\left\{\begin{aligned}}{\end{aligned}\right.}

\renewcommand{\P}{\mathbb{P}}
\renewcommand{\O}{\mathcal{O}}
\newcommand{\C}{\mathcal{C}}
\newcommand{\F}{\mathcal{F}}

\newcommand{\Z}{\mathbb{Z}}

\newcommand{\X}{\mathcal{X}}

\newcommand{\G}{\mathbb{G}_m}

\newcommand{\Pic}{{\rm Pic}}

\newcommand{\Bsm}{\mathbb{B}_{sm}(2,2g+2)}

\newcommand{\Asm}{\mathbb{A}_{sm}(2,2g+2)}

\renewcommand{\H}{\mathcal{H}_g}

\newcommand{\D}{\mathcal{D}_{2g+2}}

\renewcommand{\det}{{\rm det}}

\begin{document}

\title{A note on families of hyperelliptic curves}
\author{Sergey Gorchinskiy and Filippo Viviani}
\address{Steklov Mathematical Institute, Gubkina str. 8, 119991 Moscow (Russia)}
\email{gorchins@mi.ras.ru}
\address{Institut f\"ur Mathematik, Humboldt Universit\"at zu Berlin, 10099 Berlin (Germany).}
\email{{\tt viviani@math.hu-berlin.de}}

\thanks{The first author was partially supported by the grants RFBR 05-01-00455,
Nsh-1987.2008.1, and INTAS no. 05-1000008-8118}
\thanks{During the preparation of this paper, the second author was supported
by a grant from Mittag-Leffler Institute in Stockholm.}


\date{19 September 2008}

\begin{abstract}
We give a stack-theoretic proof for some results on families of
hyperelliptic curves.
\end{abstract}
\maketitle

\section{Introduction}

Let $k$ be a field and $g$ be an integer such that
${\rm char}(k)\ne 2$ and $g\ge 2$. All schemes that we consider
are of finite type over $k$.

Any family $\F\to S$ of smooth genus $g$ hyperelliptic curves is a
double cover of a conic bundle $\C\to S$ branched at a Cartier
divisor $D$ finite and \'etale of degree $2g+2$  over the base $S$
(see \cite{LK}). Conversely, starting with a family $(\C\to S,D)$ as above,
one can ask what are the obstructions to the
existence of a corresponding family of hyperelliptic curves
$\F\to S$ and how many such families does there exist.
The classical theory of double covers immediately gives the answer
to this question in terms of the functions on $\C$ and its Picard group $\Pic(\C)$.

In Theorem \ref{fam-hyper} we give a different answer to these questions
in terms of the geometry of the base $S$. Our proof is completely
stack-theoretic and uses the fact that the stack $\H$ of
hyperelliptic curves is a $\mu_2$-gerbe over the stack $\D$ of conic bundles endowed with an
effective Cartier divisor finite and \'etale of degree $2g+2$, and the fact that
both these stacks have an explicit description as quotient stacks (see \cite{AV} and
\cite{GV}).


As an application of the Theorem \ref{fam-hyper}, we give a proof of
two classical facts on families of hyperelliptic curves.

In Proposition \ref{univ-fam},
we prove that there exists a tautological family of hyperelliptic curves over
a non-empty open subset of the coarse moduli space $H_g$ if and only if $g$ is odd.
Moreover, we give a different proof of \cite[Thm. 3.12]{GV}, stating that such a family
never exists over the open subset $H_g^0$ corresponding to curves without extra-automorphisms
apart from the hyperelliptic involution (this is in contrast with the fact that a tautological
family exists over the open subset $M_g^0\subset M_g$ of general curves of genus
$g\geq 3$ without automorphisms).
From this result and the rationality of $H_g$ (see \cite{Bog} and \cite{Kat}),
we deduce that the stack $\H$
is rational if and only if $g$ is odd (Corollary \ref{ratio-stack}).

In Proposition \ref{MRprop}, we give a different (and for us
simpler) proof of a result of Mestrano--Ramanan (\cite{MR}), stating
that a global $g^1_2$ for a family of hyperelliptic curves exists
only in the case $g$ even.

\section{Notations}

By $\H$, $\D$, and $H_g$ denote the stack of families of
genus $g$ smooth hyperelliptic
curves, the stack of conic bundles together with an
effective Cartier divisor finite and \'etale of degree $2g+2$ over
the base, and the common coarse moduli space of two stacks above, respectively.

Recall that given a $k$-scheme $X$ and a $k$-group scheme $G$ acting
on $X$, the quotient stack, denoted as $[X/G]$, is the category fibered
in groupoids over the category of $k$-schemes, whose fiber over a
$k$-scheme $S$ is the groupoid whose objects are $G$-torsors $E\to S$ endowed with a
$G$-equivariant morphism $E\to X$ and whose arrows are isomorphisms of the above
objects. In the particular case where $X={\rm Spec}(k)$, we get the classifying stack of $G$,
denoted with $BG$, whose fiber over $S$ is the groupoid of $G$-torsors $E\to S$.

The stacks $\H$ and $\D$ admit the following description as quotient stacks
(see \cite[Cor. 4.7]{AV} and \cite[Prop. 3.4]{GV}):
$$\begin{sis}
& \H= [\Asm/(GL_2/\mu_{g+1})], \\
& \D= [\Asm/(GL_2/\mu_{2g+2})]=[\Bsm/PGL_2],\\
& H_g= \Bsm/PGL_2,\\
\end{sis}$$
where $\Asm$ is the linear space of degree $2g+2$ binary forms
without multiple roots, $\Bsm$ is the projectivization of $\Asm$,
and $GL_2$ acts on $\Asm$ by the formula $A\cdot f(x)= f(A^{-1}\cdot
x)$.

We briefly recall the notion of the rigidification of a stack
(see \cite[Section 5.1]{ACV}).
Let $\X$  be an algebraic stack over $k$ (even though everything can be
extended to a general base scheme), $H$ a
commutative $k$-group scheme and assume that for every object $\xi\in \X(T)$
 there is an embedding $H_T \subset  {\rm Aut}_T(\xi)$ compatible with pullbacks.
Then there is an algebraic stack $\X^H$ (called the rigidification of $\X$
along $H$) together with a smooth morphism of
algebraic stacks $\phi:\X\to \X^H$ uniquely determined by the properties:
\begin{enumerate}[(i)]
 \item
For any object $\xi\in \X(T)$ with image $\eta:=\phi(\xi)\in \X^H(T)$,
we have that $H(T)$ lies in the kernel of
${\rm Aut}_T(\xi) \to {\rm Aut}_T(\eta)$.
\item
The morphism $\X\to \X^H$ is universal for morphisms of stacks $\X\to \mathcal{Y}$
satisfying (i) above.
\end{enumerate}
Moreover, a moduli space for $\X$ is also a moduli space for $\X^H$ and $\X$ is a
$H$-gerbe over $\X$, which means that (see \cite{GIR} or \cite{LMB})
\begin{enumerate}[(a)]
\item The structure morphism $\phi:\X\to \X^H$ is surjective.
\item The diagonal $\Delta:\X \rightarrow \X \times_{\X^H} \X$ is surjective.
\end{enumerate}

Let $\Psi:\H\to\D$ be the morphism of stacks sending
a family $\F\to S$ of smooth genus $g$ hyperelliptic curves over $S$ into
the underlying conic bundle $\C\to S$ together with its relative Cartier
branch divisor $D\subset \C$.
By the above explicit description, it follows that
$\Psi:\H\to\D$ realizes the stack $\D$ as the $\mu_2$-rigidification of the stack
$\H$ along the hyperelliptic involution acting on families
of hyperelliptic curves.
Thus $\H$ is a $\mu_2$-gerbe over $\D$ and they have the same coarse moduli space $H_g$.

Let $H_g^0\subset H_g$ be the open subset corresponding to hyperelliptic curves
without extra-automorphisms apart from the hyperelliptic involution.
The preimage $\Bsm^0$ of $H_g^0$ in $\Bsm$
is exactly the locus where the action of $PGL_2$ is free, hence we have
$[\Bsm^0/PGL_2]=:\D^0\cong H_g^0$.

For a small category $\mathcal{E}$, by $|\mathcal{E}|$ denote the
set of isomorphism classes of objects in $\mathcal{E}$.

\section{Main statement}

This section is devoted to the proof of the following

\begin{teo}\label{fam-hyper}
For a scheme $S$, consider a family $(\C\stackrel{p}\longrightarrow
S,D)\in Ob(\D(S))$; denote by
$\Psi^{-1}(\C\stackrel{p}\longrightarrow S,D)$ the preimage in
$|\H(S)|$ of the class of $(\C\stackrel{p}\longrightarrow S,D)$ in
the set $|\D(S)|$. We have:
\begin{itemize}
\item[(i)]
If $g$ is odd, then the set $\Psi^{-1}(\C\stackrel{p}\longrightarrow S,D)$
is non-empty if and only if \\ $p_*(\omega_{\C/S}^{g+1}(D))$ is
$2$-divisible in $\Pic(S)$. If $g$ is even, then the set
$\Psi^{-1}(\C\stackrel{p}\longrightarrow S,D)$
is non-empty if and only if $p_*(\omega_{\C/S}^{g+1}(D))$ is $2$-divisible in $\Pic(S)$,
the family $p:\C\to S$ is the projectivization of
a rank two vector bundle $V\to S$, and
$\det V=p_*(\omega^{-1}_{\C/S}(-2))$ is $2$-divisible in $\Pic(S)$.
\item[(ii)]
If the set $\Psi^{-1}(\C\stackrel{p}\longrightarrow S,D)$ is non-empty, then
it is a homogeneous space for the group
$H^1_{\acute e t}(S,\mu_2)$ with respect to the following action:
an element of $H^1_{\acute e t}(S, \mu_2)$,
corresponding to a double \'etale
cover $\widetilde{S}\to S$, sends the hyperelliptic family $(\F\to S)$ to the
family $((\F\times_S \widetilde{S})/(i\times j)\to \widetilde{S}/j=S)$,
where $i$ is the hyperelliptic involution of $\F$ and $j$ is
the non-trivial automorphism of $\widetilde{S}$ over $S$. Moreover, if
$(\C\stackrel{p}\longrightarrow S,D)\in Ob(\D^0(S))$, then the action of
$H^1_{\acute e t}(S,\mu_2)$ on $\Psi^{-1}(\C\stackrel{p}\longrightarrow S,D)$
is free.
\end{itemize}
\end{teo}

The proof of the theorem uses a general result on quotient stacks.
Let $G$ be a smooth group scheme acting on a scheme $X$. The natural
morphism of quotient stacks $[X/G]\to BG=[{\rm Spec}(k)/G]$ induces a map of
sets $|[X/G](S)|\to |BG(S)|=H^1_{\acute e t}(S,G)$ for any scheme $S$, where
we used the fact that, since $G$ is a smooth group scheme, the
isomorphism classes of $G$-torsor over $S$ are parametrized
by the first \'etale cohomology group of $S$ with coefficients in $G$
(see \cite{MIL} or \cite{GIR}).

Suppose that we are given a central extension of smooth group schemes
$$
1\to K\to G\to H\to 1 \eqno (*)
$$
and an action of $H$ on a scheme $X$. Consider the ``restriction''
of this action to $G$. Then for each scheme $S$ the group
$H^1_{\acute e t}(S,K)$ acts naturally on the set $|[X/G](S)|$
by the formula $\{U_{\alpha},f_{\alpha},g_{\alpha\beta}\}\mapsto
\{U_{\alpha},f_{\alpha},g_{\alpha\beta}k_{\alpha\beta}\}$, where
$\{U_{\alpha}\}$ is an \'etale covering of $S$, the collection
$\{f_{\alpha}:U_{\alpha}\to X,g_{\alpha\beta}:U_{\alpha}\times_S
U_{\beta}\to G\}$ represents an element from $|[X/G](S)|$, and
a 1-cocycle $\{k_{\alpha\beta}\}$ represents an element from
$H^1_{\acute e t}(S,K)$.

\begin{lem}\label{exactstack}
\hspace{0cm}
\begin{itemize}
\item[(i)]
For a scheme $S$, the natural map $|[X/G](S)|\to
|[X/H](S)|$ defines a bijection between
$|[X/G](S)|/H^1_{\acute e t}(S,K)$ and the preimage of the
trivial cohomology class under the composition $|[X/H](S)|\to
H^1_{\acute e t}(S,H) \stackrel{\partial}{\longrightarrow}
H^2_{\acute e t}(S,K)$.
\item[(ii)]
If the action of $H$ on $X$ is free, then
the action of the group $H^1_{\acute e t}(S,K)$ on the set
$|[X/G](S)|$ is free.
\end{itemize}
\end{lem}
\begin{proof}
The proof of (i) is a direct check, which
uses the exact sequence of pointed sets
$$
1\to H^1_{\acute e t}(S,G)/H^1_{\acute e t}(S,K)\to H^1_{\acute e t}(S,H)\to
H^2_{\acute e t}(S,K).
$$
In order to prove (ii), suppose that $\{U_{\alpha},f_{\alpha},g_{\alpha\beta}\}$ and
$\{k_{\alpha\beta}\}$ represent elements from
$|[X/G](S)|$ and $H^1_{\acute e t}(S,K)$ such that
$\{U_{\alpha},f_{\alpha},g_{\alpha\beta}\}$ is equivalent to
$\{U_{\alpha},f_{\alpha},$ $g_{\alpha\beta}k_{\alpha\beta}\}$ in $|[X/G](S)|$.
Then, after passing to a subcovering, we see that there exists a collection
$\{g_{\alpha}:U_{\alpha}\to G\}$ such that $f_{\alpha}=g_{\alpha}f_{\alpha}$
and $g_{\alpha\beta}=g_{\alpha}^{-1}g_{\alpha\beta}k_{\alpha\beta}g_{\beta}$.
If the action of $H$ on $X$ is free, then we have $g_{\alpha}:U_{\alpha}\to K$
for all $\alpha$ and therefore the class of the cocycle $\{k_{\alpha\beta}\}$ in
$H^1_{\acute e t}(S,K)$ is trivial.
\end{proof}

\begin{proof}[Proof of Theorem \ref{fam-hyper}]
We use Lemma \ref{exactstack} with $X=\Asm$,
$K=\mu_2$, $G=GL_2/\mu_{g+1}$, $H=GL_2/\mu_{2g+2}$.

Using the explicit description
of the stack of hyperelliptic families as a quotient stack
(see \cite[Rmk. 3.3 and Thm. 4.1]{AV}),
one deduces that the ``stack-theoretic'' action of $H^1_{\acute et}(S, \mu_2)$
on $|\H(S)|$ coincides with the one described in the statement of the theorem.
This implies Theorem \ref{fam-hyper}(ii).

To prove Theorem \ref{fam-hyper}(i), we compute explicitly the obstruction map.
Recall that there is an
isomorphism of algebraic groups
$GL_2/\mu_{2g+2}\cong\G\times PGL_2$, given by the formula
$[A]\mapsto (\det(A)^{g+1},[A])$.
Using \cite[Prop. 4.6]{GV}, it is easy to show that the map
$|\D(S)|\to H^1_{\acute e t}(S,GL_2/\mu_{2g+2})$ sends
the class of a family $(\C\stackrel{p}\longrightarrow S,D)$ to the pair
$(p_*(\omega_{\C/S}^{-g-1}(-D)),\C\stackrel{p}\longrightarrow S)
\in H^1_{\acute e t}(S,GL_2/\mu_{2g+2})=
H^1_{\acute e t}(S,\G)\times H^1_{\acute e t}(S,PGL_2)$.

For $g$ odd, the isomorphism $GL_2/\mu_{g+1}\cong \G\times PGL_2$, given by
$[A]\mapsto ({\rm det}(A)^{\frac{g+1}{2}},$ $[A])$, shows that
the exact sequence $(*)$ coincides with the exact sequence
$$
1\to \mu_2\to \G\times PGL_2\stackrel{(2,1)}\longrightarrow \G\times PGL_2\to 1.
$$
For $g$ even, the isomorphism $GL_2/\mu_{g+1}\cong GL_2$, given by
$[A]\mapsto \det(A)^{\frac{g}{2}}A$, shows that the exact sequence
$(*)$ coincides with the exact sequence
$$
1\to \mu_2\to GL_2\stackrel{(\det,[\cdot])}\longrightarrow \G\times PGL_2\to 1.
$$
Hence in both cases the composition $\Pic(S)=H^1_{\acute e t}(S,\G)\to
H^1_{\acute e t}(S,GL_2/\mu_{2g+2})$ $\to H^2(S,\mu_2)$ is equal to the
coboundary map arising from the Kummer exact sequence for $\mu_2$. Thus this
composition vanishes at $p_*(\omega_{\C/S}^{-g-1}(-D))\in \Pic(S)$ if and only if
$p_*(\omega_{\C/S}^{-g-1}(-D))$ is $2$-divisible in $\Pic(S)$.

Further, for $g$ odd,
the composition $H^1_{\acute e t}(S,PGL_2)\to
H^1_{\acute e t}(S,GL_2/\mu_{2g+2})\to H^2(S,\mu_2)$ is trivial. For $g$ even,
this composition is equal to the coboundary map arising from the exact
sequence
$$
1\to \mu_2\to SL_2\to PGL_2\to 1.
$$
This coboundary map vanishes at the conic bundle $\C\stackrel{p}\longrightarrow S$
if and only if $\C=\P(V)$ for a rank two vector bundle $V\to S$ such that
$\det(V)=p_*(\omega^{-1}_{\C/S}(-2))$ is $2$-divisible in $\Pic(S)$.
This concludes the proof of Theorem \ref{fam-hyper}.
\end{proof}

\section{Examples and applications}

First let us discuss the conditions of Theorem~\ref{fam-hyper}.


Using the $2$-divisibility conditions in Theorem~\ref{fam-hyper}(i),
one can easily deduce the $2$-divisibility of the line bundle
$\O_{\C}(D)$ in the Picard group $\Pic(\C)$. The converse seems not
to be quite trivial for $g$ even. More precisely,
Theorem~\ref{fam-hyper}(ii) implies the following result.

\begin{cor}\label{cor-divis}
Let $\P(V)\to S$ be the projectivization of a rank two vector bundle
$V$ on $S$ and let $D\subset \P(V)$ be an effective Cartier divisor
finite and \'etale over $S$ of degree \mbox{$d\equiv 2 \pmod 4$}.
Suppose that $\O_{\P(V)}(D)$ is $2$-divisible in $\Pic(\P(V))$. Then
$\det(V)$ must be $2$-divisible in $\Pic(S)$.
\end{cor}

In Corollary~\ref{cor-divis} the hypothesis $D$ being \'etale over
$S$ is necessary as is shown the following example.

\begin{exa}
Consider the Hirzebruch surface $p:\mathbb{F}_n=\P(\O_{\P^1}\oplus \O_{\P^1}(-n))=\P(V)
\to \P^1=S$ for some $n\geq 0$. By \cite[V.2.18]{Har}, there exists an irreducible smooth curve
$D$ in the linear system $|d c_1 + a f|$ if $d, a>0$, where $c_1$ denotes the first Chern class of the line
bundle $\O_{\mathbb{F}_n}(1)$ and $f$ denotes the class of a fiber of $p$.
Then $\O_{\mathbb{F}_n}(D)$ is $2$-divisible in $\Pic(\mathbb{F}_n)$ if
$d$ and $a$ are even, while $\det(V)=\O_{\P^1}(-n)$ is not $2$-divisible in $\Pic(\P^1)$
if $n$ is odd. Clearly, $D$ is not \'etale over $S=\P^1$.
\end{exa}

In Theorem~\ref{fam-hyper}(i), all the divisibility
conditions are necessary as it is shown by the following two examples.

\begin{exa}
Suppose that ${\rm char}(k)$ does not divide $2g+2$. Consider a
general divisor $F\subset \P^1\times \P^2$ of bi-degree $(2g+2,1)$.
Let $R\subset \P^2$ be the ramification curve of the map
$p:F\to\P^2$ induced by the second projection. We put
$S=\P^2\backslash R$, $\C=\P^1\times S$, $D=F|_S$,
$V=\O_S\oplus\O_S$. Then the map $p:D\to S$ is \'etale of degree
$2g+2$ and $D\subset \C$ is a Cartier divisor since $\C$ is smooth.
Clearly, $\det V=\O_S$ is $2$-divisible in $\Pic(S)$. Further,
$\O_{\C}(D)=\O_{\C}(2g+2)\otimes p^*(\O_{S}(1))$. Moreover, $R$ is
irreducible and it is easily shown that $R$ has the even degree
$4g+2$, hence $\O_S(1)$ is not $2$-divisible in $\Pic(S)$.
Therefore, $p_*(\omega_{\C/S}^{g+1}(D))$ is not $2$-divisible in
$\Pic(S)$ (we use that $\omega_{\C/S}(2)=p^*(\det V)^{-1}=\O_{\C}$).
\end{exa}

\begin{exa}
For simplicity, suppose that ${\rm char}(k)=0$. Consider the blow-up
$P$ of $\P^3$ at a point $x\in \P^3$. By $\sigma:P\to\P^3$ and
$p:P\to \P^2$ denote the corresponding natural maps and by $E\subset
P$ denote the exceptional divisor. Recall that $P\cong\P(V)$, where
$V=\O_{\P^2}\oplus\O_{\P^2}(1)$. Let $T\subset \P^3$ be a general
surface of degree $2g+3$ such that $T$ contains $x$. Let $R\subset
\P^2$ be the ramification curve of the degree $2g+2$ map $p:F\to
\P^2$, where $F=\sigma^*(T)-E$. We put $S=\P^2\backslash R$,
$\C=P|_S$, $D=F|_S$. Then the map $p:D\to S$ is \'etale of degree
$2g+2$ and $D\subset \C$ is a Cartier divisor since $\C$ is smooth.
Further, $R$ is irreducible and it is easily shown that $R$ has the
even degree $(2g+3)(2g+2)-2$, hence $\det V=\O_S(1)$ is not
$2$-divisible in $\Pic(S)$. Moreover,
$\O_{\C}(D)\equiv\O_{\C}(2g+2)\otimes
p^*(\O_{S}(-1))\mod{p^*(2\Pic(S))}$, therefore
$p_*(\omega_{\C/S}^{g+1}(D))$ is $2$-divisible in $\Pic(S)$ if $g$
is even (we use that $\omega_{\C/S}(2)=p^*(\det
V)^{-1}=p^*(\O_S(-1))$).
\end{exa}

In Theorem~\ref{fam-hyper}(ii) the action of $H^1_{\acute e
t}(S,\mu_2)$ on hyperelliptic families over $S$ is not free in the
presence of extra-automorphisms as is shown by the following
example.

\begin{exa}
Let $S={\rm Spec}(A)$ with $A=k[T,T^{-1}]$, let $g$ be even, and let
$Q(x)\in k[x]$ be a degree $g+1$ polynomial without multiple roots.
We put
$$
P(x)=Q(x)Q(T/x)x^{g+1}\in A[x].
$$
Consider the family of
hyperelliptic curves $\F$ over $S$ whose affine model is given by
$\{y^2=P(x)\}$ (the corresponding family $\C$ equals to $\P^1\times
S$). Then the double \'etale cover $\widetilde{S}={\rm
Spec}(A[\sqrt{T}])\to S$ sends $\F$ to the family $\F'$ whose affine
model is given by $\{Ty^2=P(x)\}$. The map sending $(x, y)$ to
$(T/x, (yT^{g/2})/x^{g+1})$ defines an isomorphism between $\F$ and
$\F'$.
\end{exa}

Now we give some applications of Theorem \ref{fam-hyper}. First of all,
note the following immediate corollary of Theorem \ref{fam-hyper}.

\begin{cor} Let $(\C\to S, D)$ be as in Theorem~\ref{fam-hyper}.
If $g$ is odd, then for a non-empty open subset $U\subset S$, there
exists a hyperelliptic family $\F\to U$, which corresponds to
$\C|_U$. If $g$ is even, then the above statement is true if and
only if $\C$ is Zariski locally trivial, that is there exists
a non-empty open subset $V\subset S$ such that
$\C|_V\cong\P^1\times V$.
\end{cor}

Let us give a solution for the Exercise 2.3 from \cite{HM} (note
that there is a small misprint there: universal should be replaced
by tautological) together with a different proof of Theorem 3.12
from \cite{GV}.

\begin{pro}\label{univ-fam}
\hspace{0cm}
\begin{itemize}
\item[(i)]
There exists a tautological family of hyperelliptic curves over a
non-empty open subset in $H_g$ if and only if $g$ is odd.
\item[(ii)] For any $g$, there does not exist a tautological family over $H_g^0$.
\end{itemize}
\end{pro}
\begin{proof}
First we prove (i). Since $H_g$ is irreducible, we may replace it by
the open subset $H_g^0$. Further, we have $\D^0\cong H_g^0$, hence
there exists a universal family $(\C_g\to H_g^0,D_{2g+2})\in
\D^0(H_g^0)$. Explicitly, the universal family $(p:\C_g\to H_g^0,
D_{2g+2})$ is is the $PGL_2$-quotient of the family
$(p_1:\Bsm^0\times \P^1\to \Bsm^0,\mathbb{D}_{2g+2})$, where $PGL_2$
acts diagonally and $\mathbb{D}_{2g+2}$ is the tautological divisor.
Since the action of $PGL_2$ is free, it follows from \cite[Chapter
1, Section 3]{GIT} that
$$
\Pic(\C_g/H_g^0)=\Pic^{PGL_2}(\P^1)=\Z\cdot \O_{\P^1}(2).
$$
Thus the conic bundle $\C_g\to H_g^0$ does not have any
line bundle of relative degree $1$. Hence it can not be the
projectivization of a rank two vector bundle on $H_g^0$ and, by Theorem
\ref{fam-hyper}(i), we get the first conclusion.

To prove (ii) we are going to show that
$p_*(\omega_{\C_g/H_g^0}^{g+1}(D_{2g+2}))$ is not $2$-divisible in
$\Pic(H_g^0)$ if $g\geq 3$, which gives the second conclusion using
again Theorem \ref{fam-hyper}(i).
The $PGL_2$-equivariant classes of the tautological divisor
$\mathbb{D}_{2g+2}$ and the relative dualizing sheaf $\omega_{p_1}$ for the trivial
family $p_1$ are given by
$$\begin{sis}
& \O_{\Bsm^0\times \P^1}(\mathbb{D}_{2g+2})=p_1^*\O_{\Bsm^0}(1)\otimes
p_2^*\O_{\P^1}(2)^{\otimes (g+1)}, \\
& \omega_{p_1}=p_2^*\O_{\P^1}(2)^{\otimes -1}.
\end{sis}$$
Using the projection formula, we deduce that
$p_*(\omega_{\C_g/H_g^0}^{g+1}(D_{2g+2}))$ is equal to
$$p_{1*}(\omega_{p_1}^{g+1}(\mathbb{D}_{2g+2}))=
\O_{\Bsm^0}(1)\in \Pic^{PGL_2}(\Bsm^0)=\Pic(H_g^0),
$$
which is not $2$-divisible, since $\Pic(H_g^0)$ is a finite cyclic
group generated by the above element and has even cardinality for
$g\geq 3$ (\cite[Cor. 3.8]{GV}).
\end{proof}

\begin{rem}
It is interesting to compare the above result with the ones in \cite{Mum} and \cite{Ran}.
In \cite[page 58]{Mum}, one can find an explicit tautological family
over $H_1^0$. In \cite{Ran}, it is proved that the
moduli space of "framed" hyperelliptic curves
(i.e. hyperelliptic curve $C$ plus a fixed double cover $C\to \P^1$),
does have a tautological family over an open subset but not globally.
\end{rem}

Proposition~\ref{univ-fam} can be re-interpreted as a result on the
rationality of the moduli stack $\H$. Following \cite[Section
4]{BH}, we say that an irreducible algebraic stack $\X$ is
rational if it has an open substack isomorphic to $X\times BG$,
where $X$ is a rational variety and $G$ is the generic isotropy
group of $\X$.

\begin{cor}\label{ratio-stack}
The stack $\H$ is rational if and only if $g$ is odd.
\end{cor}
\begin{proof}
The open substack $\H^0\subset \H$ is a $\mu_2$-gerbe over $H_g^0$.
Since $H_g$ (and hence $H_g^0$) is rational (see \cite{Bog},
\cite{Kat}), $\H^0$ is rational if and only if it is a neutral gerbe
locally in the Zariski topology of $H_g^0$. This is equivalent to
the existence of a tautological family Zariski-locally on $H_g^0$
and hence we conclude by Proposition~\ref{univ-fam}.
\end{proof}

The other application concerns the existence of a global $g^1_2$ for
a hyperelliptic family $\F\stackrel{\pi}\longrightarrow S$, i.e.,
the existence of a line bundle on $\F$ such that its restriction to
any geometric fiber of $\pi$ coincides with the unique line bundle of degree
$2$ and having two independent global sections. We will use the following criterion
for the existence of a global $g^1_2$.

\begin{lem}\label{Criterion}
Let  $\pi:\F\to S$ be a family of hyperelliptic curves which is a double cover
of the conic bundle $p:\C\to S$.
Assume $S$ is irreducible with generic point $\eta={\rm Spec}(k(S))$.
Consider the following conditions:
\begin{itemize}\label{criterion}
\item[(i)] There exists a $g^1_2$ on $\F$.
\item[(ii)] The hyperelliptic curve $\F_{\eta}$ admits a $g^1_2$ defined over
$k(S)$.
\item[(iii)] The conic $C_{\eta}$ is isomorphic to $\P^1_{k(S)}$.
\item[(iv)] $p:\C\to S$ is the projectivization of a rank two vector bundle on $S$.
\end{itemize}
Then the following implications hold true:
(i) $\Rightarrow$ (ii) $\Leftrightarrow$ (iii) $\Leftarrow$ (iv). Moreover, if $S$ is
smooth over $k$, then the above conditions are all equivalent.
\end{lem}
\begin{proof}
The implications (i) $\Rightarrow$ (ii) and (iii) $\Leftarrow$ (iv) are clear.
Let us prove the equivalence (ii) $ \Leftrightarrow$ (iii).
Call $h$ the map from $\F_{\eta}$ to $C_{\eta}$. Now, if $\C_{\eta}\cong
\P^1_{k(S)}$ then $h^*(\O_{\P^1_{k(S)}}(1))$ provides the required
$g^1_2$ on $\F_{\eta}$. Conversely, if the $g^1_2$ of $\F_{\eta}$
is defined over $k(S)$ then $V:=\pi_*(g^1_2)=H^0(g^1_2)$ is a
vector space over $k(S)$ of dimension $2$ and, by construction,
$\C_{\eta}\cong \P(V)=\P^1_{k(S)}$.

Assume now that $S$ is smooth over $k$. Let us prove the implication
(i) $\Leftarrow$ (ii). The hypothesis implies that there is an open subset
$U\subset S$ such that $\pi^{-1}(U)\to U$ admits a $g^1_2$.
Since $S$ and $\pi$ are smooth (and hence also $\F$), we can extend the above
line bundle to a line bundle, call it $L$, on $\F$ (simply take the closure of the
Cartier = Weyl divisor associated to it). The line bundle $L$
has vertical degree $2$ everywhere since the
vertical degree is locally constant and $S$ is irreducible, and moreover
$h^0(\F_s,L_{|\F_s})\geq 2$ for every geometric point $s$ of $S$, by semicontinuity
of $h^0$. This implies that $L$ is the required $g^1_2$ on $\F$.

Let us finally prove the implication (iii) $ \Rightarrow$ (iv)
assuming that $S$ is smooth. The hypothesis implies that there exists
an open subset $U\subset S$ such that $p^{-1}(U)\to U$ admits a line
bundle of vertical degree one. As before, using that $\C$ is smooth
(since $p$ and $S$ are smooth), we can extend this line bundle to a line
bundle, call it $M$, on $\C$ that will have vertical degree one.
Since the geometric fibers of $p$ are $\P^1$,
we have that $p_*(M)$ is a
locally free sheaf of rank $2$. The natural map
$p^*(p_*(M))\rightarrow M$ is surjective since
its restriction to every geometric fiber is surjective. Hence it
determines an $S$-map $\Phi:\C\rightarrow \P(p_*(M))$
that, being an isomorphism on the geometric fibers, is an isomorphism.
\end{proof}

\begin{pro}\label{MRprop}
If $g$ is odd, then there does not exist a global $g^1_2$ for any
tautological family over a non-empty open subset in $H_g$. If $g$ is
even, then a global $g^1_2$ exists for any family of genus $g$
hyperelliptic curves over an irreducible smooth $k$-scheme.
\end{pro}
\begin{proof}
If $g$ is odd then, from the proof of Proposition \ref{univ-fam}, we known
that the universal conic bundle $\C_g\to H_g^0$ is not Zariski locally
trivial and therefore we conclude by the implication (i) $\Rightarrow$ (iii) of
the above Lemma \ref{Criterion}.

If $g$ is even and $\F\to S$ is a family of hyperelliptic curves
realized as a double cover of the conic bundle $\C\to S$,
then Theorem \ref{fam-hyper}(i) gives that $\C\to S$ is the projectivization of a
rank two vector bundle on $S$. If, moreover, $S$ is irreducible and smooth over $k$,
then there exists a global $g^1_2$ on $\F$ by the implication (iv) $\Rightarrow$
(i) of the above Lemma \ref{Criterion}.
\end{proof}

\end{document}